\documentclass{amsart}
\usepackage{amsmath, amssymb, amsthm}
\usepackage{hyperref}
\usepackage{xcolor}
\usepackage{graphicx}
\usepackage{soul}
\usepackage{amsaddr}

\newtheorem{theorem}{Theorem}[section]
\newtheorem{proposition}[theorem]{Proposition}

\newtheorem{corollary}[theorem]{Corollary}
\theoremstyle{definition}
\newtheorem{definition}[theorem]{Definition}

\title[Gradient systems and asymmetric relaxations in Riemannian Geometry]{Gradient systems and asymmetric relaxations in view of Riemannian geometry}
\author[Bravetti et al.]{Alessandro Bravetti$^1$, Miguel Ángel García Ariza$^{2,\star}$, José Roberto Romero-Arias$^2$}
\address{$^1$School of Science and Technology, University of Camerino, Camerino, Italy}
\address{$^2$Instituto de Investigaciones en Matemáticas Aplicadas y en Sistemas,\\ Universidad Nacional Autónoma de México, Mexico City, Mexico}
\thanks{$\star$\texttt{miguel.garcia@iimas.unam.mx}}


\begin{document}

\begin{abstract}

    In dually flat manifolds, there is a deep connection between gradient flows and pregeodesics. This was one of the many important contributions of Amari to information geometry. In this paper, we extend the study of this relationship to general Riemannian manifolds. Our result does not impose conditions of flatness on the connection or symmetry on its non-metricity tensor, thus broadening the geometric setting beyond Hessian manifolds. Within this framework, we provide a criterion for comparing relaxation along two different gradient descent curves of a function, formulated in terms of the non‑metricity tensor of a connection for which the gradient curves are pregeodesics. We use it to study Gaussian chains, whose relaxation trajectories coincide with gradient descent curves in the space of Gaussian distributions.
    Thus, we recover a recent result that establishes a universal asymmetry: warming up is faster than cooling down. 
    Our work illustrates how geometric insights rooted in Amari’s legacy offer new perspectives for optimization problems and stochastic processes.

\end{abstract}

\maketitle

\begin{flushright}
{\it This paper is dedicated to Professor Shun-ichi Amari in honor of his 90th birthday}
\end{flushright}

\section{Introduction}\label{sec:intro}

There is a classical result in mechanics, known as the Jacobi-Maupertuis Theorem, that relates the trajectories of a system to geodesics \cite{Abraham1967Foundations,arnold1989mathematical,pettini2007geometry}. 
It states that, 
given a Riemannian manifold $(M,g)$ and a mechanical Hamiltonian 
$$H=\frac{1}{2}||p||_g^2+V(q)$$ on $T^*M$, the projections to $M$ of
the trajectories of the Hamiltonian system at fixed energy $E$
are reparametrized geodesics of the Jacobi metric $$g^\text{J}=2(E-V(q))g\,.$$
%Clearly the Jacobi metric is only one of many possibilities of a metric that has this property.\\

The Jacobi–Maupertuis theorem can be restated in terms of connections: the Hamiltonian trajectories 
at fixed energy are reparameterized geodesics of the Levi‑Civita connection of the Jacobi metric $g^\text{J}$.
Since the Hamiltonian vector field $X_H$ is the \emph{symplectic gradient} of $H$, defined by $$\omega(X_H,\cdot)=\mathrm{d}H\,,$$ 
the theorem provides a  connection that turns symplectic‑gradient curves into geodesics.

Analogously, the Riemannian gradient of a function $f$ is defined by 
\begin{equation*}
g(\operatorname{grad} f,\cdot)=\mathrm{d} f\,.
\end{equation*}
A natural question then arises: given a Riemannian manifold $(M,g)$ and a function $f$, can we always find a connection that turns gradient curves into geodesics?

Fujiwara and Amari gave a partial answer to this question showing that on dually flat (Hessian) manifolds, 
the gradient flow of a canonical divergence from a fixed point follows 
pregeodesics of the corresponding flat connection \cite{fujiwara_gradient_1995}.
  
Despite the limitation of this result, coming from the assumption of dual-flatness, it has inspired continued interest, with recent applications spanning Weyl geometry, Hamiltonian dynamics, and non‑equilibrium physics \cite{wada_eikonal_2021,wada_huygens_2023,wada_weyl_2023,chanda_mechanics_2024,wada_hamiltonian_2024,wada_onsagers_2025}. It has also oriented the quest to find canonical divergences on non-dually flat manifolds (see e.g.~\cite{ay_novel_2015} and the discussion therein).

Because the state space of a thermodynamic system is dually flat, we can use Fujiwara and Amari’s result to study asymmetric relaxations in this context. By asymmetric relaxations we mean processes in which a system evolves faster from one initial state than from another to a common final equilibrium state, even when both initial states are equally distant from it \cite{lapolla_faster_2020,van2021toward,ibanez2024heating,tejero2025asymmetries,teza2026speedups}. Using the canonical divergence as the measure of distance to the final state, we showed that the faster relaxation curve is the one for which the Amari–Chentsov tensor evaluated along the curve is smaller whenever the speeds coincide \cite{bravetti_asymmetric_2025}. This endowed the Amari–Chentsov tensor with a clear dynamical meaning, yet the result remained confined to the dually flat setting.

%\textcolor{red}{Inspired by this result,} we turned to the problem of asymmetric relaxations in thermodynamics \cite{bravetti_asymmetric_2025},%
%where a system evolves to equilibrium faster from one initial state than from another, even when both are equally distant from equilibrium~\cite{lapolla_faster_2020,granada,others}. 
%Considering that \textcolor{red}{the space of states of} thermodynamic systems are examples of Hessian manifolds for which the canonical divergence 
%has precisely the meaning of the distance to equilibrium,
%we provided a geometric criterion to compare two gradient-descent curves \textcolor{red}{this} divergence on a dually flat manifold. The faster curve is the one for which the Amari–Chentsov tensor \textcolor{red}{along the curve} is smaller whenever the speeds coincide. This gave the Amari–Chentsov tensor a clear dynamical role, yet it remained confined to the dually flat setting of Fujiwara and Amari’s theorem.

This naturally raises a second question:
%\begin{center}
    Can we generalize this criterion beyond dually flat manifolds?
%    makes the gradient curves of a function reparameterized geodesics?
%\end{center}

Guided by these questions, we present two main results. First, we construct a connection whose pregeodesics are the gradient curves of a prescribed function. This extends Fujiwara and Amari’s theorem beyond dual flatness, and provides the foundation for the second result. Using this connection, we derive an asymmetry criterion valid on any Riemannian manifold, expressed in terms of the connection's non‑metricity tensor. 

We illustrate the scope of this extended framework 
by applying it to recover the universal warming‑cooling asymmetry in Gaussian chains observed 
by Lapolla and Godec \cite{lapolla_faster_2020},
thus obtaining a simplified proof of a physically-relevant relaxation asymmetry.
Moreover, this example shows that our criterion can be applied beyond the dually flat case.

This work is a continuation of the conversation between geometry and dynamics that Fujiwara and Amari started, where stochastic processes become geodesic motions and where the question of why warming differs from cooling finds a geometric answer. We have organized it as follows. First, we review previous results in Sect.~\ref{sec:gradientflows}. 
Then, we present our main results in Sects.~\ref{sec:asymmetricRM} and~\ref{sec:asymmetricrelaxations}.
In Sect.~\ref{sec:GaussianChains}, we apply this criterion to Gaussian chains, which are systems of particular physical relevance, and share our conclusions and perspectives in Sect.~\ref{sec:conclusion}.

\section{Gradient flows on dually flat manifolds}\label{sec:gradientflows}

We review here previous results to provide some context for our own. In what follows, $(M,g,\nabla,\nabla^*)$ denotes an $n$-dimensional dually flat manifold, meaning that 
\begin{equation*}
    X[g(Y,Z)]=g(\nabla_XY,Z)+g(Y,\nabla^*_XZ),
\end{equation*}
with both $\nabla$ and $\nabla^*$ flat and symmetric. 

The metric on a dually flat manifold is special, since it can locally be written as the Hessian of a function $\phi$:
\begin{equation*}
    g_{ij}=\frac{\partial^2\phi}{\partial\theta^i\partial\theta^j}.
\end{equation*}
This is why dually flat manifolds are also known as Hessian manifolds. The $\theta$ are affine parameters of $\nabla$, this is, the Christoffel symbols of $\nabla$ in these coordinates are zero.

There is something analogous for the affine parameters $\eta$ of the dual connection $\nabla^*$:
\begin{equation*}
    g^{ij}=\frac{\partial^2\psi}{\partial\eta_i\partial\eta_j},
\end{equation*}
which are also the components of the inverse matrix of $(g_{ij})$. The two metric potentials 
$\phi(\theta)$ and $\psi(\eta)$
are related via a Legendre transform: 
\begin{equation*}
    \phi(\theta)+\psi(\eta)=\theta^i\eta_i.
\end{equation*}

Amari and Nagaoka \cite{amari_methods_2007} proved that, to any flat connection, 
we may associate its $\nabla$-divergence: a non-negative two-point function on $M$ that vanishes only on the diagonal, given by
\begin{equation*}
    D_\nabla(p||q):=\phi(p)+\psi(q)-\theta^i(p)\eta_i(q).
\end{equation*}
When $M$ is the parameter space of an exponential family of distributions (which is dually flat with the Fisher metric and the so-called exponential and mixture connections), $D_\nabla$ coincides with the Kullback-Leibler divergence between two distributions. 

The remarkable property of $D_\nabla$ that Fujiwara and Amari \cite{fujiwara_gradient_1995} proved is that its gradient flows are ``straight lines'' as seen by $\nabla$ and $\nabla^*$. To be more precise, if we let 
$$D_q:=D_\nabla(q||\cdot) \qquad \text{and} \qquad D_q^*:=D_\nabla(\cdot||q)\,,$$ 
then
\begin{equation*}
    \nabla_{\operatorname{grad}D_q}\operatorname{grad}D_q=\operatorname{grad}D_q\qquad \text{and}\qquad \nabla^*_{\operatorname{grad}D_q^*}\operatorname{grad}D_q^*=\operatorname{grad}D_q^*.
\end{equation*}

That property allowed us to portray Newton's Law of Cooling as 
a gradient flow in a previous work \cite{bravetti_asymmetric_2025}. 
There, we also showed that the non-metricity tensor of the connections, 
known as the \textit{Amari-Chentsov tensor} in this context, 
is instrumental in determining which of two gradient descent curves will reach equilibrium first. 
Specifically, let $\gamma^1$ and $\gamma^2$ be two gradient descent curves 
(this is, two integral curves of $-\operatorname{grad} D_q$)
starting at the same distance to $q$ as measured by $D_q$ 
(namely, such that $D_q(\gamma^1(0))=D_q(\gamma^2(0))$). 
If $\nabla g(\dot\gamma_1,\dot\gamma_1,\dot\gamma_1)<\nabla g(\dot\gamma_2,\dot\gamma_2,\dot\gamma_2)$
whenever $||\dot\gamma^1||_g=||\dot\gamma^2||_g$,
then $\gamma_1$ will always be closer to $q$ than $\gamma_2$,
\textit{i.e.}, $D_q(\gamma_1(t))<D_q(\gamma_2(t))$ for all $t>0$.

Not much has been done in the case where $\nabla$ (and hence $\nabla^*$) is not flat. 
Matumoto proved \cite{matumoto_any_1993} that a dual manifold has infinitely many associated divergences, in the sense that the metric and the dual connections can be derived from them (as is the case of the $\nabla$-divergence). 
Ciaglia \textit{et al.}~proposed 
a Hamilton-Jacobi approach to define a natural divergence from a dynamical point of view \cite{ciaglia2017hamilton}.
They also argued that the non-dually-flat case is extremely important, 
since the space of pure states of a finite-level
quantum system does not admit a dually flat structure.
Following a different path, Amari and Ay found a generalized ``canonical divergence'' that reduces to $D_\nabla$ in the flat case \cite{ay_novel_2015}, but misses other nice properties like satisfying $D(p||q)=D^*(q||p)$, or having gradient flows follow pregeodesics.

This coincidence of gradient curves and pregeodesics is precisely what gave the Amari-Chentsov tensor its significance in our asymmetry criterion. It is also sufficient to compare relaxation speeds on any Riemannian manifold: if we can find a connection whose pregeodesics are the gradient curves of a function, then a similar criterion using the non-metricity tensor becomes possible, regardless of flatness. In the next section, we prove that such a connection always exists.

%This last condition is necessary to apply our above result on asymmetric relaxations, and most importantly, it is also sufficient. This observation will allow us to extend it beyond the realm of dually flat manifolds.

\section{Gradient flows follow pregeodesics on Riemannian manifolds\label{sec:asymmetricRM}}

The result of Fujiwara and Amari establishes that, on a dually flat manifold, the gradient flow of a canonical divergence coincides with pregeodesics of its associated flat connection. Can we always find a connection that makes the gradient curves of a given function pregeodesic on an arbitrary Riemannian manifold? We answer by constructing such a connection explicitly.

Let $f$ be a smooth function
on a Riemannian manifold $(M,g)$. 
We seek a connection $\nabla^f$ for which every gradient descent curve of $f$ is a pregeodesic, i.e.,
\begin{equation}\label{eq:gradientpregeodesic}
    \nabla^f_{\operatorname{grad}f}\operatorname{grad}f=\lambda \operatorname{grad }f,
\end{equation}
for some real function $\lambda$.

A good starting point is writing $\tilde\nabla^f_XY=\nabla^g_XY-A(X,Y)Z$, where $\nabla^g$ is the metric connection, 
$A$ is a (0,2)-tensor, and $Z$ some vector field to be determined. 
We can even make it symmetric by choosing $A$ symmetric, 
and a straightforward guess with what we have at hand is $A(X,Y)=g(X,Y)$. Then, the only missing ingredient is $Z$, which we may obtain from  
Eq.~\eqref{eq:gradientpregeodesic}:
\begin{equation}\label{eq:Z}
    \nabla^g_{\operatorname{grad f}}\operatorname{grad} f-||\operatorname{grad }f||^2Z=\lambda\operatorname{grad }f.
\end{equation}
This equation determines $Z$ uniquely wherever $\operatorname{grad}f \neq 0$. Consequently, we obtain the following general result:

\begin{theorem}\label{thm:main}
    For any function $f$, there exists a symmetric connection $\tilde\nabla^f$ defined on $M$ minus the critical points of $f$ by
    \begin{equation}\label{eq:nablaf}
        \tilde\nabla^f_XY:=\nabla^g_XY- 
        g(X,Y)\frac{\nabla^g_{\operatorname{grad}f}\operatorname{grad}f-\lambda\operatorname{grad }f}{||\operatorname{grad }f||^2},
    \end{equation}
    for some real function $\lambda$ and satisfying Eq.~\eqref{eq:gradientpregeodesic}. 
    Hence, every gradient curve of $f$ is a pregeodesic of $\tilde\nabla^f$.
\end{theorem}

The construction in Eq.~\eqref{eq:nablaf} is just a particular solution to the problem of finding a connection satisfying Eq.~\eqref{eq:gradientpregeodesic} for a given function $f$. Others may exist, but their classification lies beyond the scope of our present work. For the moment, we only define them formally:

    \begin{definition}
        We call any connection $\nabla^f$ satisfying
        Eq.~\eqref{eq:gradientpregeodesic} a \emph{straightening connection of $f$}.
        Equivalently, we say that $\nabla^f$ \emph{straightens} $f$.
    \end{definition}

Notice that in the definition of $\tilde\nabla^f$ we may take $\lambda\equiv0$, 
so that $\operatorname{grad}f$ is actually a geodesic of the connection. 
This will be useful to simplify some calculations later. 

It is important to remark that, although every gradient curve of $f$ is a pregeodesic of $\tilde \nabla^f$, the converse is not true.
Indeed, at any point $x\in M$ there are infinitely many geodesics (one for each tangent direction), 
while there is only one gradient curve.

When $f$ is the canonical divergence $D_\nabla(q||\cdot)$ of a flat connection (or its dual), the connection $\tilde\nabla^f$ constructed above generally does not coincide with $\nabla$. However, its defining condition given by Eq.~\eqref{eq:gradientpregeodesic} 
is the essential property for the asymmetry criterion we will develop later, where the non-metricity tensor is the main character. Actually, the non‑metricity tensor 
$$\tilde C^f(W,X,Y) := \tilde \nabla^f_W g(X,Y)$$ 
need not be totally symmetric. A direct calculation gives
\begin{equation}\label{eq:nonmetricity}
    \tilde C^f(W,X,Y) = g(W,X)g(Y,Z) + g(W,Y)g(X,Z),
\end{equation}
where $Z$ solves Eq.~\eqref{eq:Z}. This lack of full symmetry contrasts with the Amari–Chentsov tensor, and illustrates how our framework goes beyond the scope of statistical manifolds (including dually flat manifolds).

%\textcolor{blue}{This result is similar in spirit and generality to the well-known result establishing that a
%trajectory of a Hamiltonian system (of mechanical type) can be rewritten as a (reparametrized) geodesic with an appropriate metric. Typically, this is  5the Jacobi metric in configuration space, but there can be also other choices~\cite{pettini2007geometry}).}

%\textcolor{blue}{Another possible comment is the following: 
%The result of Fujiwara and Amari only holds on dually flat manifolds \emph{and only if} $f$ is the corresponding canonical divergence. In this sense, when we restrict to functions $f$ having a unique global minimum $q$, but otherwise general, the function $\tilde f(x):=f(x)-f(q)$ corresponds to $D_q:=D_{\nabla^f}(\cdot||q)$, where $D_{\nabla^f}$ would be \emph{the canonical divergence} of the dual manifold $(M,g,\nabla^f,(\nabla^f)^*)$, which is non-flat in general. Having a quick look at the discussion here~\cite{felice2021towards} it seems that we might actually have a case (See the discussion at the beginning of page 6 in the arXiv version).}\textcolor{red}{Suena muy interesante; me gustaría discutirlo con más detalle para saber bien qué escribir al respecto.}
Theorem \ref{thm:main} provides a recipe to straighten any function $f$. This can be viewed as a partial answer to an inverse problem: while the classical result of Fujiwara and Amari constructs gradient flows that are geodesics given a dually flat structure, our theorem allows us to derive a connection from a function, so that the pregeodesic property of $\operatorname{grad} f$ is automatically guaranteed.
In the special case where $f$ has a unique global minimum $q$, the shifted function $\tilde f(x):=f(x)-f(q)$ 
is always non-negative and vanishes only at $q$. 
This makes it tempting to consider $\tilde f$ as ``generating''
%study the conditions under which such
%functions are ``induced'' by a divergence $D$, in the sense that $D(x,q)=\tilde f(x)$.
%Such a divergence may be then regarded as a ``canonical divergence'', 
%\st{or in the sense of Felice and Ay \cite{felice2021towards}, as an analogue of Hamilton's principal function,} 
%since it generates 
the dualistic structure $(M,g,\tilde\nabla^f,(\tilde\nabla^f)^*)$. %, even when this structure is not dually flat. 
This perspective is reminiscent of Matumoto's approach \cite{matumoto_any_1993}, who showed that any statistical manifold can be derived from a divergence. 
Our work could then extend this idea to a broader setting, where the function $f$ is not necessarily a divergence, but induces a connection that straightens it. %, offering a new link between mechanics and geometry.

In dually flat manifolds, for any submanifold $S$,
the $\nabla$-geodesic from a point $q$ to a minimizer $\hat{p} \in S$ of the canonical divergence $D_q$ meets $S$ orthogonally. 
Thus, minimization of $D_q$ over $S$ coincides with orthogonal projection along $\nabla$-geodesics. The later extension of canonical divergences by Ay and Amari obeys this so-called \textit{geodesic projection property} under two assumptions: uniqueness of $\nabla$-geodesics connecting any two points, and the condition that the inverse exponential map at $q$ be proportional to $\operatorname{grad} D_q$ \cite{ay_novel_2015}. 

\begin{figure}
    \centering
    \includegraphics[width=\linewidth]{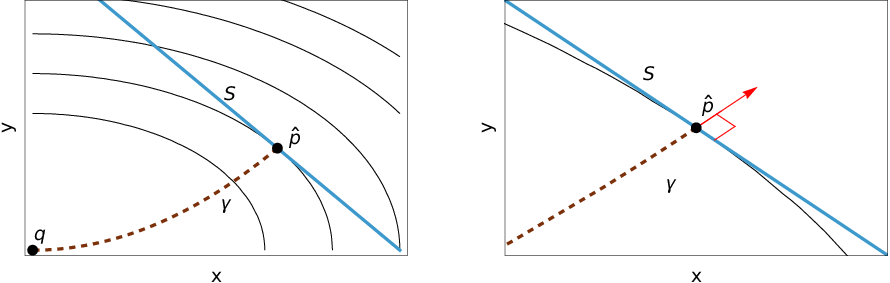}
    \caption{Geodesic projection property. The $\nabla^f$-geodesic $\gamma$ from $q$ to $\hat p$ (left) meets the submanifold $S$ orthogonally at the minimizer $\hat p$ of $f$ on $S$ (right).}
    \label{fig:gpp}
\end{figure}

The same holds for every $\nabla^f$ whose geodesics connecting any two points are unique, because any connection that straightens $f$ has an inverse exponential map proportional to $\operatorname{grad}f$ (see Fig. \ref{fig:gpp}). We provide an explicit proof below.

    \begin{proposition}\label{prop:geodesicprojection}
    Let $\nabla^f$ be a connection that straightens $f$, such that the geodesic connecting any two points is unique. 
    Then $\nabla^f$ has the geodesic projection property.
    \end{proposition}
\begin{proof}
    If $\imath:S\hookrightarrow M$ is a submanifold and $f\circ\imath$ attains a minimum at $\hat p\in S$, then
\begin{equation*}
    \mathrm{d}(f\circ\imath)_{\hat p} =0.
\end{equation*}
This means that for any vector $v$ tangent to $S$ at $\hat p$, 
\begin{equation*}
    0= 
    \mathrm{d}(f\circ\imath)_{\hat p}(v)
    =\mathrm{d}f_{\hat p}(v)=g_{\hat p}(\operatorname{grad}f_{\hat p},v),
\end{equation*}
where we are denoting by the same $\hat p$ and $v$ their images under $\imath$ and $\mathrm{d}\imath_{\hat p}$, respectively. 

Finally, since the $\nabla^f$-geodesic connecting $q$ and $\hat p$ is unique, its velocity must be proportional to $\operatorname{grad}f$ because the connection straightens $f$, and so the geodesic meets $S$ orthogonally at $\hat p$.
\end{proof}

%When $f$ has a unique global minimum at $q$ and the $\nabla^f$-geodesic between any two points is unique, the minimizer of $f$ on a submanifold $S$ is precisely the orthogonal projection of $q$ onto $S$ along $\nabla^f$-geodesics. In this way, Eq. \eqref{eq:gradientpregeodesic} preserves the nice geometric relationship between a connection and its canonical divergence.

%\st{This condition not only allows us to associate a connection to a function, but also embodies the distinctive geometric properties of a canonical divergence.}

In what follows, we will use Theorem~\ref{thm:main}
to extend our previous characterization of asymmetric relaxations \cite{bravetti_asymmetric_2025} to general Riemannian manifolds.

\section{Asymmetric relaxations on Riemannian manifolds}\label{sec:asymmetricrelaxations}
%\st{In Riemannian geometry, two radial pregeodesics ending at a point $q$ and starting at the same distance from it are indistinguishable in their relaxation. The reason is that, if we  parametrize them so that $\nabla_{\dot\gamma}\dot\gamma=-\dot\gamma$, their speed decays as $||\dot\gamma(t)||=||\dot\gamma(0)||e^{-t}$. Hence the remaining distance to $q$ is simply $||\dot\gamma(0)||e^{-t}$ for both. This uniformity reflects the isotropy of the Levi‑Civita connection, for which all directions from $q$ are equivalent, at least locally.
%This isotropy is lost when $\nabla$ is a non-metric connection.}

In previous work, we showed that on dually flat manifolds, the component of $\nabla g$ along gradient descent curves of the canonical divergence determines which of two such curves relaxes first, provided that both start equally ``far'' from $q$ \cite{bravetti_asymmetric_2025}. There, we measured distance with the canonical divergence itself, and we used the fact that its gradient descent curves are $\nabla$-pregeodesics. This fact alone suffices to obtain an analogous asymmetry criterion in a much broader setting.

Before stating our result, notice that considering gradient descent curves $\gamma$ 
relaxing to the unique minimum $q$ of a function $f$ provides a natural notion of ``distance'' to $q$, given by $f(\gamma(t))$,  
assuming, without loss of generality, that $f(q)=0$.
This motivates the following definition.

    \begin{definition}\label{def:fequidistant}
       Let $\gamma_1$ and $\gamma_2$ be two integral curves of the gradient descent of $f$.   
       We say that their initial conditions are \emph{$f$-equidistant from $q$, or simply $f$-equidistant,} 
       if $f(\gamma_1(0))=f(\gamma_2(0))$. 
       Moreover, if $f(\gamma_1(t))\leq f(\gamma_2(t))$ for all $t\geq 0$, then we say that \emph{$\gamma_1$ is faster than $\gamma_2$}.
    \end{definition}

Now we can state our second main result.

\begin{theorem}\label{thm:asymmetric}
     Let $f$ be a real function on $M$ attaining its unique
     minimum at $q$. Let $\nabla^f$ straighten $f$ with $\lambda\equiv\text{const.}$ 
     Let $\gamma_1$ and $\gamma_2$ be two integral curves of the gradient descent of $f$ 
     with initial conditions $f$-equidistant from $q$.
    If 
    \begin{equation}
        C^f(\dot\gamma_1,\dot\gamma_1,\dot\gamma_1)<C^f (\dot\gamma_2,\dot\gamma_2,\dot\gamma_2),
    \end{equation}
    whenever $||\dot\gamma_1||=||\dot\gamma_2||$,
    then $\gamma_1$ is faster than $\gamma_2$.
\end{theorem}

\begin{proof}
    We want to see that, with the above hypotheses, $\Delta f:=f(\gamma_2(t))-f(\gamma_1(t))\geq0$, for all $t\geq 0$. Given that $f$ is sufficiently well behaved, this would follow if all critical values of $\Delta f$ are maxima, since $\Delta f(0)=\lim_{t\to\infty}\Delta f=0$. 
    Denote by $\gamma$ an integral curve of the gradient descent of $f$. Then $\dot\gamma=-\operatorname{grad}f$ and
    \begin{equation}\label{eq:fdot}
        \dot f=\mathrm{d}f(\dot\gamma)=g(\operatorname{grad}f,\dot\gamma)=-||\dot\gamma||^2.
    \end{equation}
    So, when $||\dot\gamma_1||=||\dot\gamma_2||$, $\Delta f$ attains a critical value, whose nature is determined by the sign of $\Delta\ddot f.$
    From the last equation, we have that
    \begin{equation}
        \ddot f=-\dot\gamma[g(\dot\gamma,\dot\gamma)]=-\nabla^f_{\dot\gamma}g (\dot\gamma,\dot\gamma)-2g(\nabla^f_{\dot\gamma}\dot\gamma,\dot\gamma)=-C^f(\dot\gamma,\dot\gamma,\dot\gamma)-2\lambda g(\dot\gamma,\dot\gamma)=-C^f(\dot\gamma,\dot\gamma,\dot\gamma)-2\lambda \dot f,\label{eq:ddotf}
    \end{equation}
    using Eq.~\eqref{eq:gradientpregeodesic} in the third equality. 
    
    Theorem~\ref{thm:main} guarantees the existence of a connection $\nabla^f$ that straightens $f$ with $\lambda\equiv\text{const.}$ For such a connection, at a critical point $t_*$ we have
    \begin{equation}
    \Delta \ddot f(t_*)=-C^f(\dot\gamma_2, \dot\gamma_2,\dot\gamma_2)\vert_{t_*}+C^f(\dot\gamma_1,\dot\gamma_1,\dot\gamma_1)\vert_{t_*}.
    \end{equation}
    The last expression is negative if and only if
    \begin{equation}
        C^f(\dot\gamma_1,\dot\gamma_1,\dot\gamma_1)\vert_{t_*}<C^f(\dot\gamma_2,\dot\gamma_2,\dot\gamma_2)\vert_{t_*},\label{eq:main}
    \end{equation}
    and $\Delta f(t_*)$ is a maximum.
    
    Thus, if \eqref{eq:main} holds every time the descent speeds are equal, the relaxation along $\gamma_1$ will be faster than that along $\gamma_2$.

\end{proof}

\begin{figure}
    \centering
    \includegraphics[width=\linewidth]{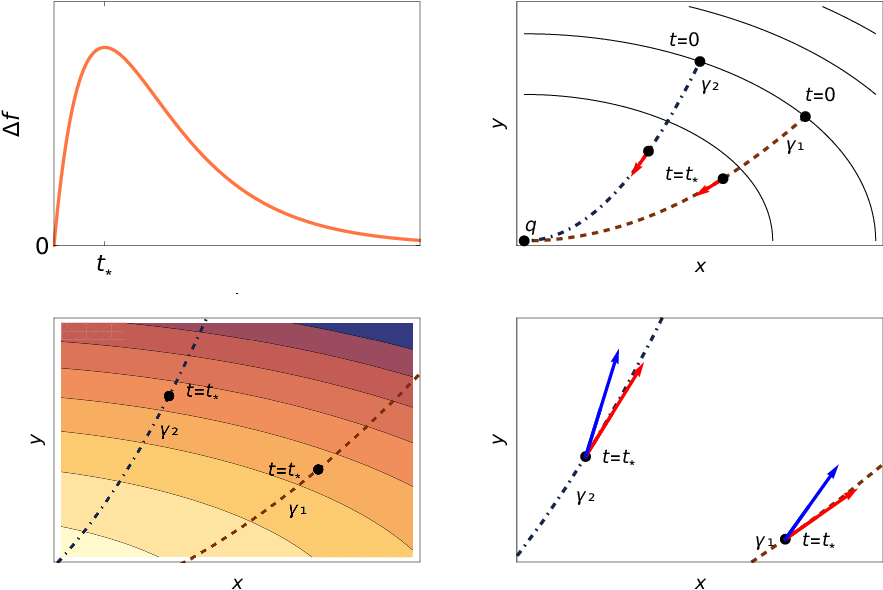}
    \caption{An illustration of our asymmetry criterion. 
    \textbf{Upper left:} $\Delta f(t)=f(\gamma_1(t))-f(\gamma_2(t))$ attains its maximum at $t=t_*$.
    \textbf{Upper right:} Two gradient descent curves $\gamma_1$ and $\gamma_2$ with $f$-equidistant initial conditions. At $t_*$,
    their velocities (red) have equal length.
    \textbf{Lower left:} Heat map of $-C^f(\operatorname{grad }f,\operatorname{grad} f,\operatorname{grad }f)$ (lighter = larger). 
    At $t_*$, $C^f(\gamma_1,\gamma_1,\gamma_1)>C^f(\gamma_2,\gamma_2,\gamma_2)$, so $\gamma_2$ relaxes faster.
    \textbf{Lower right:} Covariant acceleration (blue) at $t_*$. 
    Its tangential component (red) is negative for both curves, larger in magnitude for $\gamma_2$, 
    explaining its faster relaxation.
    }
    \label{fig:asymmetric}
\end{figure}

Equation \eqref{eq:ddotf} expresses $C^f(\dot\gamma,\dot\gamma,\dot\gamma)$  in terms of $f$ and $\lambda$ only. 
Consequently, \textit{any} straightening connection of $f$
shares this component of its non-metricity along the gradient flow. 
In particular, a connection $\nabla^f$ of which the gradient curves are geodesics ($\lambda\equiv0$) 
simplifies Eq.~\eqref{eq:ddotf} to 
    \begin{equation}\label{eq:fddotCf}
    \ddot f=-C^f(\dot\gamma,\dot\gamma,\dot\gamma)\,.
    \end{equation}

We may also express the non-metricity tensor along gradient curves in terms of the curves. Using the Levi‑Civita connection, 
\begin{equation*}
    \ddot f= -2g(\nabla^g_{\dot\gamma}\dot\gamma,\dot\gamma).
\end{equation*}
Substituting into \eqref{eq:ddotf} yields an equivalent formula that involves only $\gamma$:
\begin{equation}
    C^f(\dot\gamma,\dot\gamma,\dot\gamma) = 2\bigl[-\lambda\|\dot\gamma\|^2 + g(\dot\gamma, \nabla^g_{\dot\gamma}\dot\gamma)\bigr]
\end{equation}
Thus, Theorem~\ref{thm:asymmetric} states that
the curve having the smallest tangential metric acceleration when the relaxation speeds coincide will always be faster (see Fig. \ref{fig:asymmetric}).

Our asymmetry criterion also tells us when asymmetry is forbidden. 
    \begin{corollary}
     If $f$ is such that one of its straightening connections is a metric connection with $\lambda\equiv\text{const.}$, 
     then the gradient curves of $f$ cannot show any asymmetry.
    \end{corollary}
    \begin{proof}
        If 
$\nabla^f$ is metric, Eq.~\eqref{eq:ddotf} reduces to
$\ddot f=-2\lambda\dot f$, and therefore we have, for all~$t$,
\begin{equation*}
    \Delta\ddot f=-2\lambda\Delta\dot f\,.
\end{equation*}
If $\lambda\equiv\text{const.}$, the conditions $\Delta f(0)=\lim_{t\to\infty}\Delta f=0$ force $\Delta f\equiv0$. 
%Hence, 
%metric pregeodesics cannot exhibit asymmetric relaxation.
%Comentario sobre esto: propongo que quitemos este último comentario, eclipsa un poco la conclusión. Pensaba ponerlo abajo, pero nuestro resultado necesita la $f$, no las geodésicas solamente. Eso que queremos decir sobre las geodésicas, creo, es lo que decimos abajo en el primer renglón (con la f).
    \end{proof}

A straightforward consequence of this is that the Riemannian distance function does not have asymmetric relaxations.

However, this result is more interesting when read in the opposite direction:
functions for which a relaxation asymmetry exists
cannot be straightened by a metric connection (with $\lambda\equiv\text{const.}$).

Symmetry, however, does not necessarily
imply metricity, because asymmetric relaxations are determined by the function $f$ only,
not by the connection used to determine them.
Indeed, we can use Theorem \ref{thm:main} to find non-metric connections
that straighten a function whose relaxation is symmetric.

%\st{a non‑metric connection may still yield symmetric relaxations for all pairs of descent curves. Setting $\Delta f\equiv0$ yields $\Delta C^f=0$ for any pair of gradient curves of $f$ originating from the same level set along their relaxation to the critical point. But this does not mean that $C^f\equiv0$. For example, take $f(x,y)=(x^2+y^2)/2$ on the Euclidean plane and $C=\mathrm{d}r^3$. Any two gradient descent curves starting on the same level curve of $f$ are lines converging to the origin. They are equivalent under rotations, so their rotation-invariant speed is the same at all times, and thus $\dot r^1=\dot r^2$ always. This yields $\Delta C=0$ along the relaxation, yet $C\not\equiv0$.}

%Here are some important remarks: iii) The result still holds if $\dot\gamma$ differs from $-\nabla^f_{\dot\gamma}\dot\gamma$ by a vector that is orthogonal to $\dot\gamma$. This basically means $\nabla^f_{\dot\gamma}\dot\gamma=\dot\gamma+X$ with $X$ tangent to the level sets of $f$. 

With Eq.~\eqref{eq:nablaf}, we can illustrate how Theorem~\ref{thm:asymmetric} 
works explicitly on a manifold that is not dually flat, which is also of physical relevance.

\section{Universal asymmetry in relaxations of Gaussian chains}\label{sec:GaussianChains}

Using our criterion, we can now prove the asymmetry in the relaxation of Gaussian chains reported by Lapolla and Godec \cite{lapolla_faster_2020}. In this context, the distance to equilibrium will be measured by the gradient potential of the dynamics, which is related to the dual Kullback–Leibler divergence.

Gaussian chains are $N+1$ beads connected by ideal springs with zero rest length. 
The dynamics of their normal modes are described by multivariate Gaussian distributions with zero means and variances $a_k$, 
which evolve according to
\begin{equation}
    a_k=2\frac{1+(\tilde T-1)e^{-2\lambda_kt}}{\lambda_k}.\label{eq:varianceLG}
\end{equation}
The $\lambda_k$ are the eigenvalues of the matrix that allows us to work with independent modes. The parameter $\tilde T$ denotes the ratio between the initial temperature of the system and the equilibrium temperature (see Supplementary Material of Ref. \cite{lapolla_faster_2020}).

An interesting feature of Gaussian chains is that they display a \textit{universal} asymmetry:
 a system warming up ($\tilde T<1$) reaches equilibrium faster than one cooling down ($\tilde T>1$) when both start equally far from equilibrium \cite{lapolla_faster_2020}, with the distance measured by the KL divergence. In what follows, we reproduce a similar result based on Theorem~\ref{thm:asymmetric}.

Eq.~\eqref{eq:varianceLG} is the solution of
\begin{equation}\label{eq:LG}
    \dot a_k=-2\lambda_k(a_k-a_k^*),
\end{equation}
with initial condition $a_k(0)=2\tilde T/\lambda_k$ and $a_k^*=2/\lambda_k$.

This is the gradient flow of 
$$F:=\sum_k^N\lambda_k\left(\frac{a_k^*}{a_k}-\ln\left(\frac{a_k^*}{a_k}\right)-1\right).$$
The metric here is the Fisher metric for multivariate normal distributions, which is block-diagonal, 
with each block corresponding to the Fisher metric of a single mode:
\begin{equation*}
    g_k=\frac{2}{a_k}\mathrm{d}\mu_k^2+\frac{1}{2a_k^2}\mathrm{d}a_k^2\,;
\end{equation*}
the gradient potential $F$ is proportional to the weighted sum of the (mixture) KL divergence of each mode,
\begin{equation*}
    D_\text{KL}(a_k^*||a_k)=\frac{1}{2}\left[\frac{a_k^*}{a_k}-\ln\left(\frac{a_k^*}{a_k}\right)-1\right].
\end{equation*}
 Since the gradient potential has a global minimum at $(a_1^*,\ldots,a_N^*)$, we can use it as a measure of distance to equilibrium.

To apply our criterion in Theorem \ref{thm:asymmetric},
we will first consider one mode of the chain.
All the modes are independent, 
so the dynamics given by Eq.~\eqref{eq:LG} is still a gradient flow,
with gradient potential $F_k:=2\lambda_kD_\text{KL}(a_k^*||a_k)$.

Using a connection that straightens $F_k$ with $\lambda\equiv0$ (see Eq.~\eqref{eq:fddotCf}) we obtain
\begin{equation*}
    -C^{F_k}(\dot\gamma,\dot\gamma,\dot\gamma)=\ddot F_k=2\lambda_k\frac{a_k^*}{a_k}\left(\frac{\dot a_k}{a_k}\right)^2,
\end{equation*}
for any gradient descent curve of $F_k$.

Now, let $\gamma^\pm=(0,a_k^\pm)$ be two such curves,
which are also solutions of Eq.~\eqref{eq:LG} given by Eq.~\eqref{eq:varianceLG}, 
having initial temperatures
$\tilde T^+>1$ and $\tilde T^-<1$, respectively.
At the critical point $t_*$ of $\Delta F_k$,
the speeds of $a_k^+$ and $a_k^-$ must be the same, meaning that
\begin{equation*}
    \left(\frac{\dot a_k^+}{a_k^+}\right)_{t_*}^2=\left(\frac{\dot a_k^-}{a_k^-}\right)_{t_*}^2=:S\,.
\end{equation*}
Plugging this into $\Delta\ddot F_k|_{t_*}$ we obtain
\begin{equation*}
    \Delta\ddot F_k|_{t_*}=-C^{F_k}(\dot\gamma^+,\dot\gamma^+,\dot\gamma^+)_{t_*}+C^{F_k}(\dot\gamma^-,\dot\gamma^-,\dot\gamma^-)_{t_*}=2\lambda_k S\left(\frac{a_k^*}{a_k^+}-\frac{a_k^*}{a_k^-}\right)_{t_*}<0\,,
\end{equation*}
because $\lambda_k>0$ and $a_k^-<a_k^*<a_k^+$ always. 
So, $\gamma^-$ is faster than $\gamma^+$, according to Theorem \ref{thm:asymmetric}.
Since this is true for every mode, 
$\sum \Delta F_k=\Delta F\geq 0$, 
and therefore the chain warms up faster than it cools down.

To make sure that we are truly beyond the dually flat case, 
we may compute the scalar curvature $s$ of $\tilde\nabla^{F_k}$ with $\lambda\equiv0$:
\begin{equation*}
    s=\frac{a_k (a_k-5 a_k^*)}{(a_k-a_k^*)^2}.
\end{equation*}

%\textcolor{blue}{
%Indeed, in principle our criterion in Theorem~\ref{thm:asymmetric} is useful to establish which of 2 initially $f$-equidistant curves relaxes the fastest.
%However, sometimes one can use this criterion to prove what we call a \emph{universal asymmetry:} suppose
%there is a relevant parameter, say the temperature $T$, characterizing the relaxation to equilibrium, say equilibrium corresponds to 
%$T=1$ and ask the following question:
%considering $f$-equidistant trajectories that start either below or above the equilibrium temperature, is there a general
%statement that trajectories warming up are always faster than those cooling down (or viceversa)?}

This example illustrates the usefulness of our criterion in Theorem~\ref{thm:asymmetric} 
to establish which of two curves starting equally away from equilibrium relaxes the fastest.
However, we have obtained a stronger result for Gaussian chains: 
any curve starting below the equilibrium temperature will relax faster 
than any other curve starting above it,
provided that they start from $F$-equidistant initial conditions. 
This is parameterized by $\tilde T$ and, in this sense, the asymmetry is \textit{universal},
and is simply stated as \textit{warming up is faster than cooling down}.

%\textcolor{blue}{Sugiero usar ``below the equilbrium temperature'' aquí arriba, porque podría haber un conflicto
%si solo decimos  ``equilibrio'' con el hecho de que el sistema se está relajando a un mínimo de $F$, lo cual también
%es un equilibrio o punto crítico de la dinámica.}

%\subsection{Riemannian geodesics are symmetric}
%\textcolor{red}{Isn't this example trivial by the fact that Riemannian geodesics are
%geodesics of the metric connection and therefore cannot exhibit asymmetry?}\\
%\textcolor{red}{If not trivial, expand this example by adding all the definitions and calculations needed to understand.}

%As a first application of our asymmetry criterion, 
%let us take $f$ as the Riemannian distance to a fixed point $q$. 
%Its gradient descent curves are geodesics parameterized by arc length. 
%Along any such curve, $f(\textcolor{blue}{\gamma(t)})=d(q,\gamma(0))-t$, and hence $\dot f=-1$ and $\ddot f=0$. 
%Eq. \eqref{eq:ddotf} then gives $C^f(\dot\gamma,\dot\gamma,\dot\gamma)=2\lambda$ for \textit{every} $\nabla^f$. 
%If additionally $\lambda\equiv\text{const.}$, $\Delta C^f=0$ for all pairs of gradient descent curves, 
%confirming that $\Delta f\equiv0$ when these curves start from the same distance to $q$. Thus, we obtain from our 
%Prop. \ref{thm:asymmetric} that the gradient descent of the Riemannian distance is symmetric.

\section{Concluding remarks}\label{sec:conclusion}

Showing the deep connection between gradient flows and pregeodesics in Hessian manifolds was part of Amari's valuable contribution to information geometry. Besides the many implications and applications of this insight, it was a fundamental component of the geometric
criterion for asymmetric relaxations that we developed previously \cite{bravetti_asymmetric_2025}. In that context, the distance to equilibrium is measured by the KL divergence, which makes sense in information geometry and thermodynamics. 
We extended this result to Riemannian geometry, considering that \textit{any} function with a unique minimum works as well. 

In dually flat manifolds, it is fundamental that the connection we use to define geodesics is flat, because pregeodesics and gradient descents of the canonical divergence coincide in that case \cite{fujiwara_gradient_1995}. However, this is no longer true for any divergence function, nor for the ``canonical'' divergence of a non-flat connection \cite{ay_novel_2015}. 
We translate this obstacle to finding a connection whose pregeodesics describe a given gradient, and we provide an explicit recipe for this in Theorem \ref{thm:main}. Of course, many more connections that straighten a given $f$ with other desirable properties may exist. Their characterization is an interesting open problem that might provide new insights to generalize canonical connections in information geometry.

There are many directions in which our result may be extended. One is to consider dynamics that are not necessarily gradient flows, for which the function $f$ in Theorem \ref{thm:main} is a Lyapunov function. 
Another is to understand if symmetry in relaxation always implies
the existence of a metric connection that straightens $f$.

Furthermore, our geometric characterization of asymmetric relaxations on 
Riemannian manifolds can be useful in optimization problems where only the initial $f$-distance 
to a minimum is known. Theorem \ref{thm:asymmetric} may help to choose an initial condition that will make the gradient descent faster.

Finally, the geometric framework developed here provides a natural setting to investigate further instances of asymmetric relaxations in thermodynamics. In particular, the Mpemba effect, both classical and quantum, offers a fertile ground for the application of these ideas \cite{biswas2023mpemba,biswas2023mpembalangevin,qian2025intrinsic,ares2025quantum,bettmann2025information}. 
Our reproduction of the asymmetric relaxation in Gaussian chains using a distance measure that is different from the Kullback–Leibler divergence, also aligns with recent work that points to a divergence-independent characterization of relaxation asymmetries \cite{van_vu_thermomajorization_2025}.

%Finally, in thermodynamics, it will be worth studying further examples of possible relaxation asymmetries using the 
%powerful tools introduced here. 
%The geometric study of the Mpemba effect (either quantum or classical) is
%also particularly interesting~\cite{biswas2023mpemba,biswas2023mpembalangevin,qian2025intrinsic,ares2025quantum,bettmann2025information}.
%Hablar aquí de "thermomajorization", de Hasegawa y Van Vu, en relación a cómo con otra medida de distancia al equilibrio reprodujimos el resultado de Lapolla y Godec, apuntando a que las relajaciones asimétricas deberían no depender de la divergencia que se use para medir distancia al equilibrio. 

These perspectives speak to the potential of the nice interplay between geometry and dynamics first illuminated by Amari and his collaborators.

\section*{Acknowledgements}
AB would like to thank Andrea Mari and Shin-itiro Goto for insightful discussions.
AB gratefully acknowledges the Simons Center for Geometry and Physics, Stony Brook University, 
and the organizers of the 2026 Program ``Contact geometry,
general relativity and thermodynamics'', 
during which some of the research for this paper was performed.

The work of MAGA was funded by SECIHTI, programa Estancias
Posdoctorales por M\'exico.  

JRRA thanks DGAPA-UNAM for the support provided by program PAPIIT Grant IN-109525.


\begin{thebibliography}{10}

\bibitem{Abraham1967Foundations}
Ralph Abraham and Jerrold~E. Marsden.
\newblock {\em {Foundations of Mechanics}}.
\newblock Mathematical Physics Monograph Series. W. A. Benjamin, 1967.

\bibitem{arnold1989mathematical}
V.I. Arnold.
\newblock {\em Mathematical methods of classical mechanics}, volume~60.
\newblock Springer, 1989.

\bibitem{pettini2007geometry}
Marco Pettini.
\newblock {\em Geometry and topology in Hamiltonian dynamics and statistical
  mechanics}.
\newblock Springer, 2007.

\bibitem{fujiwara_gradient_1995}
Akio Fujiwara and Shun-ichi Amari.
\newblock Gradient systems in view of information geometry.
\newblock {\em Physica D: Nonlinear Phenomena}, 80(3):317--327, 1995.

\bibitem{wada_eikonal_2021}
Tatsuaki Wada, Antonio~M. Scarfone, and Hiroshi Matsuzoe.
\newblock An eikonal equation approach to thermodynamics and the gradient flows
  in information geometry.
\newblock {\em Physica A: Statistical Mechanics and its Applications},
  570:125820, 2021.

\bibitem{wada_huygens_2023}
Tatsuaki Wada, Antonio~M. Scarfone, and Hiroshi Matsuzoe.
\newblock Huygens' equations and the gradient-flow equations in information
  geometry.
\newblock {\em International Journal of Geometric Methods in Modern Physics},
  20(14):2450012, 2023.

\bibitem{wada_weyl_2023}
Tatsuaki Wada.
\newblock A {Weyl} geometric approach to the gradient-flow equations in
  information geometry, 2023.
\newblock arXiv:2212.14706 [math-ph].

\bibitem{chanda_mechanics_2024}
Sumanto Chanda and Tatsuaki Wada.
\newblock Mechanics of geodesics in {Information} geometry and {Black} {Hole}
  {Thermodynamics}.
\newblock {\em International Journal of Geometric Methods in Modern Physics},
  21(05):2450098, 2024.

\bibitem{wada_hamiltonian_2024}
Tatsuaki Wada and Antoni~M. Scarfone.
\newblock A {Hamiltonian} approach to the gradient-flow equations in
  information geometry, 2024.
\newblock arXiv:2406.11224 [math-ph].

\bibitem{wada_onsagers_2025}
Tatsuaki Wada and Antonio~Maria Scarfone.
\newblock Onsager’s {Non}-{Equilibrium} {Thermodynamics} as {Gradient} {Flow}
  in {Information} {Geometry}.
\newblock {\em Entropy}, 27(7):710, 2025.

\bibitem{ay_novel_2015}
Nihat Ay and Shun-ichi Amari.
\newblock A {Novel} {Approach} to {Canonical} {Divergences} within
  {Information} {Geometry}.
\newblock {\em Entropy}, 17(12):8111--8129, 2015.

\bibitem{lapolla_faster_2020}
Alessio Lapolla and Aljaž Godec.
\newblock Faster {Uphill} {Relaxation} in {Thermodynamically} {Equidistant}
  {Temperature} {Quenches}.
\newblock {\em Physical Review Letters}, 125(11):110602, 2020.

\bibitem{van2021toward}
Tan Van~Vu and Yoshihiko Hasegawa.
\newblock Toward relaxation asymmetry: Heating is faster than cooling.
\newblock {\em Physical Review Research}, 3(4):043160, 2021.

\bibitem{ibanez2024heating}
Miguel Ib{\'a}{\~n}ez, Cai Dieball, Antonio Lasanta, Alja{\v{z}} Godec, and
  Ra{\'u}l~A Rica.
\newblock Heating and cooling are fundamentally asymmetric and evolve along
  distinct pathways.
\newblock {\em Nature Physics}, 20(1):135--141, 2024.

\bibitem{tejero2025asymmetries}
{\'A}lvaro Tejero, Rafael S{\'a}nchez, Laiachi~El Kaoutit, Daniel Manzano, and
  Antonio Lasanta.
\newblock Asymmetries of thermal processes in open quantum systems.
\newblock {\em Physical Review Research}, 7(2):023020, 2025.

\bibitem{teza2026speedups}
Gianluca Teza, John Bechhoefer, Antonio Lasanta, Oren Raz, and Marija Vucelja.
\newblock Speedups in nonequilibrium thermal relaxation: Mpemba and related
  effects.
\newblock {\em Physics Reports}, 1164:1--97, 2026.

\bibitem{bravetti_asymmetric_2025}
Alessandro Bravetti, Miguel~Ángel García~Ariza, and Pablo Padilla.
\newblock Asymmetric relaxations through the lens of information geometry.
\newblock {\em Journal of Physics A: Mathematical and Theoretical},
  58(12):125004, 2025.

\bibitem{amari_methods_2007}
Shun-ichi Amari and Hiroshi Nagaoka.
\newblock {\em Methods of information geometry}.
\newblock Number 191 in Translations of mathematical monographs. American
  Mathematical Society, USA, 2007.

\bibitem{matumoto_any_1993}
Takao Matumoto.
\newblock Any statistical manifold has a contrast function---on the
  ${C}^3$-functions taking the minimum at the diagonal of the product manifold.
\newblock {\em Hiroshima Mathematical Journal}, (2):327--332, 1993.

\bibitem{ciaglia2017hamilton}
Florio~M Ciaglia, Fabio Di~Cosmo, Domenico Felice, Stefano Mancini, Giuseppe
  Marmo, and Juan~M P{\'e}rez-Pardo.
\newblock Hamilton-jacobi approach to potential functions in information
  geometry.
\newblock {\em Journal of Mathematical Physics}, 58(6), 2017.

\bibitem{biswas2023mpemba}
Apurba Biswas, VV~Prasad, and R~Rajesh.
\newblock Mpemba effect in driven granular gases: Role of distance measures.
\newblock {\em Physical Review E}, 108(2):024902, 2023.

\bibitem{biswas2023mpembalangevin}
Apurba Biswas, R~Rajesh, and Arnab Pal.
\newblock Mpemba effect in a langevin system: Population statistics,
  metastability, and other exact results.
\newblock {\em The Journal of chemical physics}, 159(4), 2023.

\bibitem{qian2025intrinsic}
Dongheng Qian, Huan Wang, and Jing Wang.
\newblock Intrinsic quantum mpemba effect in markovian systems and quantum
  circuits.
\newblock {\em Physical Review B}, 111(22):L220304, 2025.

\bibitem{ares2025quantum}
Filiberto Ares, Pasquale Calabrese, and Sara Murciano.
\newblock The quantum mpemba effects.
\newblock {\em Nature Reviews Physics}, 7(8):451--460, 2025.

\bibitem{bettmann2025information}
Laetitia~P Bettmann and John Goold.
\newblock Information geometry approach to quantum stochastic thermodynamics.
\newblock {\em Physical Review E}, 111(1):014133, 2025.

\bibitem{van_vu_thermomajorization_2025}
Tan Van~Vu and Hisao Hayakawa.
\newblock Thermomajorization {Mpemba} {Effect}.
\newblock {\em Physical Review Letters}, 134(10):107101, 2025.

\end{thebibliography}
\end{document}